\documentclass[10pt]{amsart}
\pdfoutput=1
\usepackage{amssymb,amscd,graphicx,enumerate,epsfig}
\usepackage[bookmarks=true]{hyperref}

\newtheorem{theorem}{Theorem}[section]

\newtheorem{proposition}[theorem]{Proposition}
\newtheorem{corollary}[theorem]{Corollary}

\newtheorem{lemma}[theorem]{Lemma}

\newcommand{\Theoremno}[2]{\noindent\textbf{Theorem #1.} \textit{#2}}

\theoremstyle{definition}
\newtheorem{definition}[theorem]{Definition}

\newcommand{\D}{\mathcal{D}}
\newcommand{\DV}{\mathcal{D}_V}
\newcommand{\DW}{\mathcal{D}_W}
\newcommand{\DVW}{\mathcal{D}_{VW}}

\newcommand{\WR}{\mathcal{WR}}
\begin{document}

\title{
 On critical Heegaard splittings of tunnel number two composite knot exteriors}
\author{Jungsoo Kim}
\date{11, Oct, 2012}

\begin{abstract}
	In this article, we prove that a tunnel number two knot induces a critical Heegaard splitting in its exterior if there are two weak reducing pairs such that each weak reducing pair contains the cocore disk of each tunnel.
	Moreover, we prove that a connected sum of two $2$-bridge knots or more generally that of two $(1,1)$-knots can induce a critical Heegaard splitting in its exterior as the examples of the main theorem.
	Finally, we give an equivalent condition for a weak reducing pair to be determined by a compressing disk uniquely when the manifold is closed, irreducible and the Heegaard splitting is of genus three and unstabilized.

\end{abstract}

\address{\parbox{4in}{
	Research Institute of Mathematics, Seoul National University\\ 
	1 Gwanak-ro, Gwanak-Gu, Seoul 151-747, Korea\\
	Tel. +82-2-880-6562, Fax. +82-2-877-8435}} 
	
\email{pibonazi@gmail.com}

\maketitle

\section{Introduction and result} 

Throughout this paper, all surfaces and 3-manifolds will be taken to be compact and orientable.
In \cite{Bachman1}, Bachman introduced the concept ``\textit{critical surfaces}'' and he proved several theorems about incompressible surfaces, the number of Heegaard splittings with respect to its genus, and the minimal genus common stabilization.
Since a critical surface has disjoint compressions on it's both sides, if the surface is a Heegaard surface, then the splitting is weakly reducible, i.e. a critical Heegaard splitting is a kind of weakly reducible splitting.
But in some aspects, it shares common properties with strongly irreducible splittings. 
For example, if the splitting is strongly irreducible or critical, then the manifold is irreducible (Lemma 3.5 of \cite{Bachman2}.) 
Indeed, the intersection of an incompressible surface $S$ and a Heegaard surface $F$ can be isotoped essential on both $S$ and $F$ if the splitting is critical or strongly irreducible (see Theorem 5.1 of \cite{Bachman1} and Lemma 6 of \cite{Schultens2}.) 
Bachman also proved Gordon's conjecture by using the series of generalized Heegaard splittings and critical Heegaard splittings (see \cite{Bachman2}.)
In his recent work \cite{Bachman2010}, he also introduced the concept \textit{``topologically minimal surfaces''}, where a strongly irreducible surface is an index $1$ topological minimal surface, and a critical surface is an index $2$ topological minimal surface, this is a way to regard strongly irreducible surfaces and critical surfaces in a unified viewpoint. 

Although critical Heegaard splittings have many powerful properties as proved in Bachman's recent works, it is not easy to determine whether a weakly reducible splitting is critical.
In \cite{JungsooKim1}, the author found a condition for an unstabilized, genus three Heegaard splitting of an irreducible manifold
to be critical by analysing the subcomplex of the disk complex for a Heegaard surface spanned by the disks where each disk among them can be contained in some weak reducing pair.

In this article, we prove the following theorem by using the method of \cite{JungsooKim1}.
\begin{theorem}(the Main Theorem)\label{the-main-theorem}
	Suppose that $M$ is an orientable, irreducible $3$-manifold and $H=(V,W;F)$ is a genus three, unstabilized Heegaard splitting of $M$.
	Suppose that there exist two weak reducing pairs $(D_0, E_0)$ and $(D_1, E_1)$.
	If both $D_0$ and $D_1$ are non-separating in $V$ and one is not isotopic to the other in $V$, then $H$ is critical.
\end{theorem}

As a direct application of the Main Theorem, we get the following corollary.

\begin{corollary}\label{corollary-tunnelnumber-two}
	Suppose that $K$ is a tunnel number two knot and $\{t_1,t_2\}$ be a tunnel system of $K$.
	Let $H=(V,W;F)$ be the induced Heegaard splitting of the exterior of $K$ by this tunnel system.
	If there are weak reducing pairs $(D_0, E_0)$ and $(D_1, E_1)$ such that $D_0$ and $D_1$ come from the cocores of $N(t_1)$ and $N(t_2)$ respectively, then $H$ is critical.
	Moreover, every incompressible surface $S$ in the exterior can be isotoped so that the intersection $S\cap F$ is essential in both $S$ and $F$.
\end{corollary}





In Appendix \ref{section-unique-wr}, we prove the following theorem.\\

\Theoremno{A.1}{
	Suppose $F$ is an unstabilized, genus three Heegaard surface in a closed, orientable, irreducible 3-manifold.
	Then either $F$ is the amalgamation of two genus $2$ Heegaard surfaces over a torus, or every compressing disk for F belongs to at most one weak reducing pair.
}\\

In Section \ref{section2}, we introduce basic definitions and properties about the disk complex and ``\textit{$n$-compatible weak reducing pairs}'' and prove Theorem \ref{the-main-theorem} and Corollary \ref{corollary-tunnelnumber-two}.
In Section \ref{section-2-bridges}, we prove a connected sum of two $2$-bridge knots induces a critical Heegaard splitting in its exterior by using the Lustig and Moriah's result in \cite{LustigMoriah}.
In Section \ref{section-1-1-knots}, more generally,  we prove a connected sum of two $(1,1)$-knots induces a critical Heegaard splitting in its exterior by using the Moriah's result in \cite{Moriah2004}.
Finally, we induce an equivalent condition for a weak reducing pair to be determined by a compressing disk when the manifold is closed in Appendix \ref{section-unique-wr}.

\section{The disk complex and the proof of the Main Theorem\label{section2}}

\begin{definition}
	Let $F$ be a surface in a compact, orientable $3$-manifold $M$.
	Then the \emph{disk complex} $\D(F)$ is defined as follows:
	\begin{enumerate}
		\item Vertices of $\D(F)$ are isotopy classes of compressing disks for $F$.
		\item A set of $m+1$ vertices forms an $m$-simplex if there are representatives for each
that are pairwise disjoint.
	\end{enumerate}
	Now we consider $F$ as a Heegaard surface of the splitting $M=(V,W;F)$.
	Let $\DV(F)$ and $\DW(F)$ be the subcomplexes of $\D(F)$ spanned by compressing disks in $V$ and $W$ respectively.
	We call these complex ``the disk complexes of $V$ and $W$''.
	McCullough proved that the disk complex of a non-trivial, non-punctured (i.e. with no boundary components homeomorphic to $2$-sphere) compression body is contractible (see \cite{McCullough1991}.)
	
	Let $\DVW(F)$ be the subcomplex of $\D(F)$ spanned by the simplices that have at least one vertex in $V$, and at least one vertex in $W$.
	In this article, we will use letter ``D'' to denote compressing disks in $V$ and ``E'' to denote those in $W$.
%
%
\end{definition}

\begin{definition}
	Let $(D_0, E_0)$, $\cdots$, $(D_n, E_n)$ be weak reducing pairs of a Heegaard splitting $H=(V,W;F)$.
	Suppose that one of the following conditions holds for every choice of $(D_i, E_i)$ and $(D_j, E_j)$ ($0\leq i\neq j\leq n$).
	\begin{enumerate}
		\item $D_i$ is isotopic to $D_j$ in $V$, $E_i\cap E_j=\emptyset$, and $E_i$ is not isotopic to $E_j$ in $W$.
		\item $E_i$ is isotopic to $E_j$ in $W$, $D_i\cap D_j=\emptyset$, and $D_i$ is not isotopic to $D_j$ in $V$.
	\end{enumerate}
	Then we call these weak reducing pairs \emph{$n$-compatible weak reducing pairs}. \footnote{An $n$-simplex of $\WR$ in \cite{JungsooKim1} is exactly the same as $n$-compatible weak reducing pairs in this article.}
	This means that
	\begin{enumerate}
		\item exactly one compressing disk for a weak reducing pair among them belongs to each of the other $n$-weak reducing pairs (the choice of ``exactly one compressing disk'' depends on the choice of the corresponding weak reducing pair among the other weak reducing pairs), and
		\item we can draw all compressing disks coming from these $(n+1)$-weak reducing pairs without intersections in $M$ up to isotopy.
	\end{enumerate}
	We define a \emph{$0$-compatible weak reducing pair} as a single weak reducing pair.
	Obviously, if there are $n$-compatible weak reducing pairs, then any choice of $k$-weak reducing pairs among them gives $(k-1)$-compatible weak reducing pairs for all $1\leq k\leq n+1$. 
\end{definition}

By the next lemma, we can identify $n$-compatible weak reducing pairs as a certain kind of $(n+1)$-subsimplex in $\DVW$.

\begin{lemma}[Lemma 2.6 of \cite{JungsooKim1}]\label{lemma2-JungsooKim2012-A}
	Suppose that there are $n$-compatible weak reducing pairs.
	Then these $(n+1)$-weak reducing pairs have the form (a) $(D_0,E)$, $\cdots$, $(D_n,E)$ or (b) $(D,E_0)$, $\cdots$, $(D,E_n)$.
	Therefore, we get an $(n+1)$-subsimplex from a connected component in $\DVW(F)$ such that one vertex comes from $\DW(F)$ and the other $(n+1)$-vertices come from $\DV(F)$ or vise versa.
Conversely, we can identify such $(n+1)$-subsimplex in $\DVW$ as $n$-compatible weak reducing pairs.
\end{lemma}

\begin{definition}[Definition 3.3 of \cite{Bachman2}]\label{definition-critical}
	Let $F$ be a surface in some $3$-manifold which is compressible to both sides.
	The surface $F$ is \textit{critical} if the set of all compressing disks for $F$ can be partitioned into subsets $C_0$ and $C_1$ such that the following hold.
	\begin{enumerate}
		\item For each $i=0,1$ there is at least one weak reducing pair $(D_i,E_i)$, where $D_i$, $E_i\in C_i$.
		\item If $D\in C_i$ and $E\in C_j$, then $(D,E)$ is not a weak reducing pair for $i\neq j$.
	\end{enumerate}
\end{definition}

Bachman also introduced the concept of \emph{topologically minimal surfaces} in \cite{Bachman2010} as follows.

\begin{definition}[Definition 1.1 of \cite{Bachman2010}]
	The \emph{homotopy index} of a complex $\Gamma$ is defined to be 0 if $\Gamma=\emptyset$, and the smallest $n$ such that $\pi_{n-1}(\Gamma)$ is non-trivial, otherwise.
	We say a surface $F$ is \emph{topologically minimal} if its disk complex $\D(F)$ is either empty or non-contractible.
	When $F$ is topologically minimal, we say its \emph{topological index} is the homotopy index of $\D(F)$.
\end{definition}

The next lemma gives a connection between critical surfaces and their topological indices.

\begin{proposition}[Theorem 2.5 of \cite{Bachman2010}]
A surface in a compact, orientable $3$-manifold has topological index 2 if and only if it is critical.
\end{proposition}

Therefore, if a surface $F$ is critical, then we can triangulate a $1$-sphere $S$ so that there exists a simplicial map $i:S\to \D(F)$ and its image is non-trivial in $\D(F)$.

\begin{definition}
	Suppose that $(D_1,E_1)$ and $(D_2, E_2)$ are $1$-compatible weak reducing pairs and $f$ be the $2$-subsimplex in $\DVW$ determined by these three compressing disks.
	Let us label $f$ ``D'' (``E'' resp.) if $E_1=E_2$ ($D_1=D_2$ resp.) and call $f$ a \textit{D-face} (an \textit{E-face} resp.)
	From now on, we will identify $1$-compatible weak reducing pairs as a $2$-subsimplex in $\DVW(F)$ determined by three compressing disks such that one of them comes from one side of $F$ and the others come from the other side of $F$.
	Moreover, we will say that two weak reducing pairs $(D, E)$ and $(D', E')$ are \emph{weakly compatible} if both weak reducing pairs are the same or there is a finite sequence of $2$-subsimplices $\Delta_0$, $\cdots$, $\Delta_m$ in $\DVW$ such that 
	\begin{enumerate}
		\item $(D,E)\subset \Delta_0$ and $(D', E')\subset \Delta_m$,
		\item  each $\Delta_i$ is a D- or an E-face, and
		\item  $\Delta_i$ and $\Delta_{i+1}$ share a weak reducing pair for $0\leq i \leq m-1$.
	\end{enumerate}
\end{definition}

In section 8 of \cite{Bachman2}, Bachman considered a sequence of compressing disks 
$$D=D_0 - E=E_0 - D_1 - E_1 - \cdots - D'=D_m - E'=E_m,$$
where (a) $D_i = D_{i+1}$ or $D_i\cap D_{i+1}=\emptyset$ and (b) $E_i = E_{i+1}$ or $E_i\cap E_{i+1}=\emptyset$ for each $0\leq i \leq m-1$ and both $(D_i, E_i)$ for $0\leq i\leq m$ and $(D_{i+1},E_i)$ for $0\leq i \leq m-1$ are weak reducing pairs.
He defined the distance between two weak reducing pairs $(D,E)$ and $(D',E')$ using the minimal length of this sequence.
If there is no such sequence between them, the distance is defined as $\infty$.

In Lemma 8.4 of \cite{Bachman2}, the distance of two different weak reducing pairs is finite when $(D,E)$ shares a compressing disk with $(D',E')$.
In particular, we can assume that $D_i\neq D_{i+1}$ or $E_i\neq E_{i+1}$ for some $0\leq i \leq m-1$ in this case, i.e. we can find the $1$-compatible weak reducing pairs $(D_i, E_i)$ and $(D_{i+1}, E_i)$, or $(D_{i+1}, E_i)$ and  $(D_{i+1}, E_{i+1})$ for some $i$ respectively.

Hence we get the next lemma immediately.

\begin{lemma}[Lemma 8.4 of \cite{Bachman2}]\label{lemma-Bachman}
	Suppose $F$ is an embedded surface in a $3$-manifold.
	 If there are different weak reducing pairs such that one shares a compressing disk with the other, then we can choose $1$-compatible weak reducing pairs for $F$.
\end{lemma}
%


The next proposition gives a connection between a critical surface and the existence of two non-weakly compatible weak reducing pairs.
	
\begin{proposition}[Proposition 2.3 of \cite{JungsooKim1}]\label{prop-Bachman}
	If there are two weak reducing pairs $(D, E)$ and $(D', E')$ such that they are not weakly compatible,  then the surface is critical.
\end{proposition}

Note that this proposition is equivalent to Lemma 8.5 of \cite{Bachman2}.

Let $\bar{D}\subset V$, $\bar{E}\subset W$ be compressing disks, where $\partial \bar{D}$ and $\partial\bar{E}$ intersect transversely in a single point.
Such a pair of disks is called a \textit{canceling pair} of disks for the splitting.
If there is a canceling pair, then we call the splitting \textit{stabilized}.

In the next part of this section, we introduce another properties of $m$-compatible weak reducing pairs when $M$ is irreducible and the Heegaard splitting $H$ is unstabilized and of genus three.

Since the only irreducible $3$-manifold which has a boundary component isomorphic to a $2$-sphere is $B^3$ and a Heegaard splitting of genus three of $B^3$ is stabilized by using Waldhausen's Theorem (see \cite{Waldhausen}), we do not consider $B^3$ and assume that a boundary component of $M$ is not homeomorphic to a $2$-sphere.

The following lemma gives the maximal number $n$ for choosing $n$-compatible weak reducing pairs if the manifold is irreducible and the splitting is unstabilized and of genus three, and describes how the boundary curves for the disks of $1$-compatible weak reducing pairs behave.

\begin{lemma}[Lemma 3.2 and Corollary 3.3 of \cite{JungsooKim1}]\label{corollary-pants}
	Suppose that $M$ is an irreducible $3$-manifold and $H=(V,W;F)$ is a genus three, unstabilized Heegaard splitting of $M$.
	Then we can choose at most $1$-compatible weak reducing pairs.
	Moreover, if $(D_0,E)$ and $(D_1, E)$ are $1$-compatible weak reducing pairs, then $\partial D_0$ is separating and $\partial D_1$ is non-separating in $F$, and $\partial D_0 \cup \partial D_1$ cuts off a pair of pants from $F$.
	In addition, if $\tilde{F}\subset F$ is obtained from the once-punctured genus two component of $F-\partial D_0$ by removing $\partial E$, then $\tilde{F}$ does not have compressing disks in $V$ other than those parallel to $D_0$.
\end{lemma}

If there are $1$-compatible weak reducing pairs, then we can imagine the situations as in Figure \ref{fig-1-simplices}.
(The left one is when $\partial E$ is non-separating and the right one is when $\partial E$ is separating in $F$.)

\begin{figure}
	\includegraphics[width=6cm]{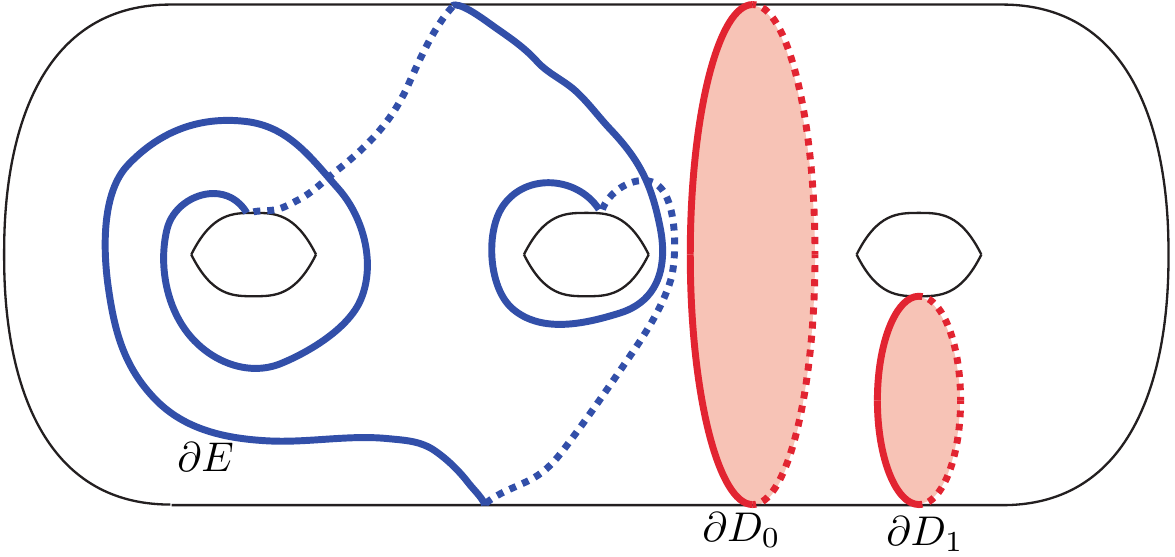}\quad\includegraphics[width=6cm]{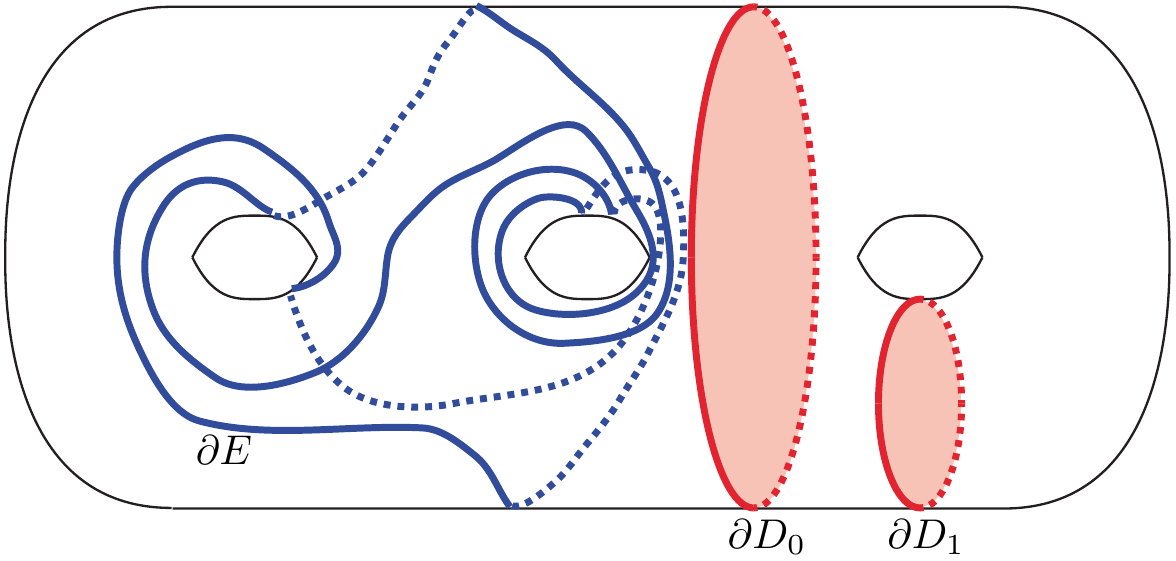}
	\caption{Two examples of $1$-compatible weak reducing pairs \label{fig-1-simplices}}
\end{figure}

The next proposition is the immediate corollary of Lemma \ref{corollary-pants}.

\begin{proposition}
	Suppose that $M$ is an irreducible $3$-manifold and $H=(V,W;F)$ is a genus three, unstabilized Heegaard splitting of $M$.
	If $F$ is topologically minimal, then the topological index of $F$ is at most four.
\end{proposition}
 
\begin{proof}
	If $F$ has topological index $n$, then $\pi_{n-1}(\D(F))$ is non-trivial, and thus there is a continuous map $i:S \to \D(F)$ which is not homotopic to a point, where $S$ is an $(n-1)$-sphere.
	Triangulate $S$ so that the map $i$ is simplicial.
	Then we can say that $i(S)$ is a finite union of $(n-1)$-cells embedded in $\D(F)$.
	Since $\DV(F)$ and $\DW(F)$ are contractible but $i(S)$ is not contractible in $\D(F)$, there must be an $(n-1)$-cell $\Delta$ spanned by vertices in $\DV$ and $\DW$ among them.
	Let the vertices of $\Delta$ be $\{D_0, \cdots, D_k, E_{k+1}, \cdots E_{n-1}\}$, where $0\leq k \leq n-2.$
	If $k\geq 2$, then we can choose $2$-compatible weak reducing pairs $(D_0, E_{k+1})$, $(D_1, E_{k+1})$, and $(D_2, E_{k+1})$, which contradicts Lemma \ref{corollary-pants}.
	Therefore, we get $k\leq 1$. 
	Similarly, if $(n-1)-(k+1) = n-k-2\geq 2$, then we can choose $2$-compatible weak reducing pairs $(D_k, E_{k+1})$, $(D_k, E_{k+2})$, and $(D_k, E_{k+3})$, i.e. a contradiction.
	Hence, we get $n-k\leq 3$.
	This means that $\dim(\Delta)\leq 3$, and we can describe $\Delta$ as $\{D_0, D_1, E_2, E_3\}$ if $\Delta$ is of dimension three by Lemma \ref{corollary-pants}.
\end{proof}

\begin{lemma}[Lemma 3.4 of \cite{JungsooKim1}]\label{lemma8}
	Assume $M$, $H$, and $F$ as in Lemma \ref{corollary-pants}.
	If there are two adjacent D-faces $f_1$ and $f_2$ in $\DVW$ such that $f_1$ is determined by $D_0$, $D_1$, and $E$, and $f_2$ is determined by $D_1$, $D_2$, and $E$, then $\partial D_1$ is non-separating, and $\partial D_0$, $\partial D_2$ are separating in $F$.
\end{lemma}

\begin{definition}
	Let us consider a $1$-dimensional graph as follows.
	\begin{enumerate}
		\item We assign a vertex to each D-face.
		\item If a D-face shares a weak reducing pair with another D-face, then we assign an edge between these two vertices in this graph.
	\end{enumerate}
	We call this graph ``the \emph{D-graph}''.
			If there is a maximal subset $\mathcal{D}$ of D-faces in $\DVW$ representing a connected component of the D-graph and the component is not an isolated vertex, then we call $\mathcal{D}$ ``a \emph{D-cluster}''.
	Similarly, we define ``the \emph{E-graph}'' and  ``an \textit{E-cluster}'' for E-faces.
	In a D-cluster (an E-cluster resp.), every weak reducing pair gives the common $E$-disk ($D$-disk resp.).
\end{definition}

We can immediately get the next lemma about a D- or an E-cluster from Lemma \ref{lemma8}.

\begin{lemma}[Lemma 3.6 of \cite{JungsooKim1}]\label{lemma-cluster}
	There is only one unique weak reducing pair in a D-cluster which can belong to two or more faces in the D-cluster.
\end{lemma}

\begin{definition} 
	By Lemma \ref{lemma-cluster}, there is a unique weak reducing pair in a D-cluster (an E-cluster resp.) adjacent to two or more faces in the cluster.
	We call it \textit{``the center of a D-cluster (an E-cluster resp.)''}.
	We call the other weak reducing pairs \textit{``hands of a D-cluster (an E-cluster resp.).''} See Figure \ref{fig-dcluster}.
	Note that if a D-face $f$ representing the $1$-compatible weak reducing pairs $(D_0, E)$ and $(D_1, E)$ is contained in a D-cluster, then one of $(D_0, E)$ and $(D_1, E)$ is the center and the other is a hand.
	Lemma \ref{lemma8} means that the D-disk from the center of a D-cluster is non-separating, and the D-disks from hands are all separating.
\end{definition}

\begin{figure}
	\includegraphics[width=4.5cm]{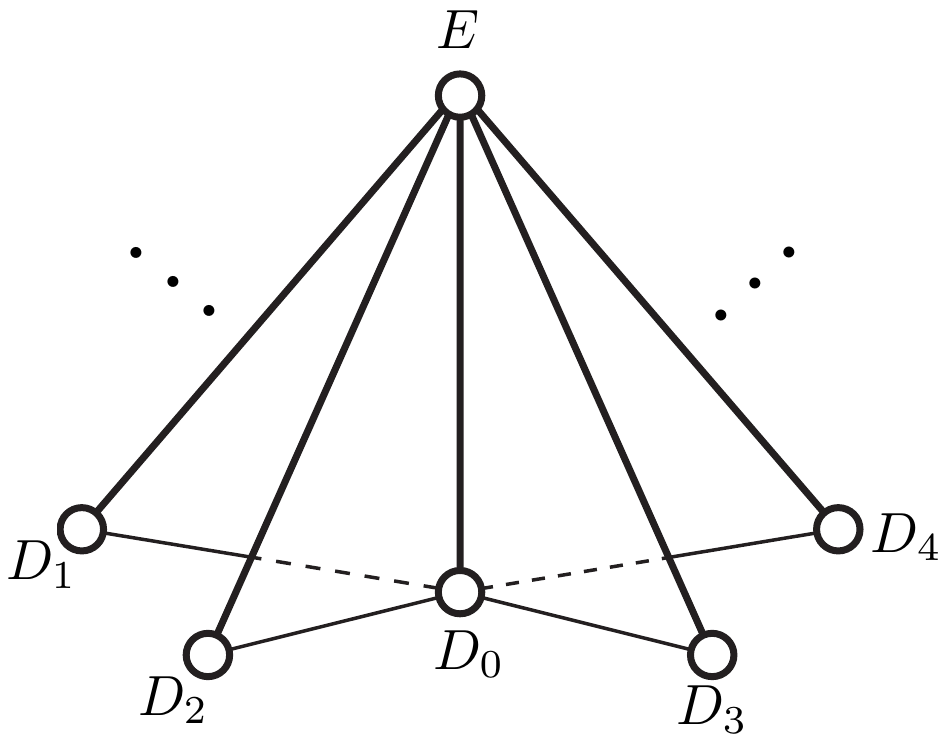}
	\caption{An example of a D-cluster in $\DVW$. $(D_0, E)$ is the center and the other weak reducing pairs are hands. \label{fig-dcluster}}
\end{figure}

\begin{lemma}[Lemma 3.8 of \cite{JungsooKim1}]\label{lemma-DorE} 
	Assume $M$, $H$, and $F$ as in Lemma \ref{corollary-pants}.
	Every D-face belongs to some $D$-cluster. 
	Moreover, every $D$-cluster has infinitely many hands.
\end{lemma}

\begin{lemma}\label{corollary-corollary-pants}
	Assume $M$, $H$, and $F$ as in Lemma \ref{corollary-pants}.
	Suppose that $(D_0, E)$ and $(D_1, E)$ are $1$-compatible weak reducing pairs.
	Let $D_0$ be the separating one and $D_1$ be the non-separating one in $V$.
	Then we get the following results.
	\begin{enumerate}
		\item $D_0$ is a band sum of two parallel copies of $D_1$ in $V$.\label{corollary-corollary-pants-1}
		\item $D_1$ is uniquely determined by $D_0$, i.e. the D-disk of the center of the D-cluster is determined by the D-disk of a hand of this cluster.\label{corollary-corollary-pants-2}
	\end{enumerate}
\end{lemma}

\begin{proof}
	By Lemma \ref{corollary-pants}, $\partial D_0\cup \partial D_1$ cuts off a pair of pants $P$ from $F$.
	If we thicken $D_1$ in $V$ slightly, then we can get two parallel copies $D_1'$ and $D_1''$ in $V$ from the boundary of the ball obtained by thickening $D_1$.
	If we isotope $P$ so that the two boundary circles of $P$ which come from $\partial D_1$ coincide with $\partial D_1'$ and $\partial D_1''$, then the disk $P\cup D_1' \cup D_1''$ is isotopic to $D_0$ in $V$. (Since a compression body is irreducible, the sphere $D_0\cup (P\cup D_1' \cup D_1'')$ must bound a $3$-ball in V.) Moreover, if we deformation retract $P$ into a shape of a pair of handcuffs (i.e. it seems like O--O) where two rings come from $\partial D_1'$ and $\partial D_1''$, then the wire connecting two rings corresponds to the arc realizing the band sum of $D_1'$ and $D_1''$.
	This completes the proof of (\ref{corollary-corollary-pants-1}).
	
	Take $V'$ from the components obtained by cutting $V$ along $D_0$ such that $F'=\partial V\cap\partial V'$ is a once-punctured torus.
	By Lemma \ref{corollary-pants}, $\partial D_0\cup \partial D_1$ cuts off a pair of pants $P$ from $F$, i.e. $\partial D_1$ belongs to $F'$.
	Since we can compress $V'$ along $D_1$ and $V$ does not have any $S^2$-component in its minus boundary, $V'$ must be a solid torus and $D_1$ is a meridian disk of $V'$, i.e. $D_1$ is uniquely determined up to isotopy.
	This complete the proof of (\ref{corollary-corollary-pants-2}).
\end{proof}

If we generalize the idea of Lemma \ref{corollary-corollary-pants}, then we get the following lemma.

\begin{lemma}[the Key Lemma]\label{key-lemma}
	Assume $M$, $H$, and $F$ as in Lemma \ref{corollary-pants}.
	Suppose that there are two different weak reducing pairs $(D_0, E_0)$ and $(D',E')$, both weak reducing pairs are weakly compatible, and $E_0$ is non-separating in $W$.
	Then $E'$ must be (a) isotopic to $E_0$ if it is non-separating or (b) a band sum of two parallel copies of $E_0$ if it is separating in $W$.
\end{lemma}

\begin{proof}
	Let $\Delta_0-\cdots-\Delta_m$ be the sequence of $2$-subsimplices in $\DVW(F)$ realizing the relation of two weak reducing pairs and assume that $m$ is such smallest integer.
	Simplify this sequence as the sequence $P=P_0-P_1-\cdots -P_n$ ($n\leq m$), where each $P_i$ is a sequence of only D-faces or only E-faces and assume that if $P_i$ consists of only D-faces, then $P_{i-1}$ and $P_{i+1}$ consists of only E-faces, or vise versa.
	Note that this sequence is ultimately the sequence of weak reducing pairs as well as $2$-subsimplices in $\DVW$.
	Assume that $(D_0, E_0)$ is the initial weak reducing pair of $P_0$ and $(D', E')$ is the terminal weak reducing pair of $P_n$.
		
	 By Lemma \ref{lemma-DorE}, each $\Delta_i$ is contained in a D- or an E-cluster for $0\leq i \leq m$, and so is each $P_i$ for $0\leq i \leq n$.
	Here, if $P_k$ is contained in a D-cluster, then $P_{k+1}$ is contained in a E-cluster for $0\leq k \leq n-1$, or vise versa.
	Moreover, since $P$ consists of minimal number of $2$-subsimplices in $\DVW(F)$ from $(D_0, E_0)$ to $(D', E')$ so that adjacent $2$-subsimplices in $P$ share a weak reducing pair, the number of D- or E-faces in each $P_i$, say $|P_i|$,  must be one or two (see Figure \ref{fig-dcluster}.)
	If the sequence $\Delta_0-\cdots-\Delta_m$ consists of only D-faces, then $E'=E_0$, i.e. the proof ends.
	Therefore, we assume that there exists at least one E-face in the sequence.

	Suppose that some $P_i$ is contained in an E-cluster.
	If $|P_i|=1$, then one E-disk for $P_i$ is non-separating in $W$ (corresponding to the center of this E-cluster) and the other $E$-disk is separating in $W$ (corresponding to a hand of the E-cluster.) by Lemma \ref{corollary-pants}.
	In this case, the separating E-disk of $P_i$ is a band sum of two parallel copies of the non-separating one by Lemma \ref{corollary-corollary-pants}.
	If $|P_i|=2$, then the initial weak reducing pair of $P_i$ (the nearest weak reducing pair of $P_i$ from $(D_0, E_0)$) and terminal weak reducing pair of $P_i$ (the farthest weak reducing pair of $P_i$ from $(D_0, E_0)$) correspond to some hands of the E-cluster, i.e. the E-disks for both weak reducing pairs are separating in $W$.
	In this case, the E-disks of the initial weak reducing pair and the terminal weak reducing pair of $P_i$ are band sums of two parallel copies of the E-disk of the center of this E-cluster by Lemma \ref{corollary-corollary-pants}.

	Suppose that $P_{i_0}$ is the first $P_i$ contained in some E-cluster.
	Let the E-disk of the initial weak reducing pair of $P_{i_0}$ be $E'_0$ and that of the terminal weak reducing pair  of $P_{i_0}$ be $E''_0$.
	If $E'_0\neq E_0$, then some E-face must precede $P_{i_0}$ in the sequence. 
	But this means that there is another $P_j$ contained in some E-cluster by Lemma \ref{lemma-DorE} for some $j<i_0$, which contradicts the choice of $P_{i_0}$.
	Therefore, $E'_0=E_0$.
	Since $E_0$ is non-separating in $W$, we get $|P_{i_0}|=1$ and $E''_0$ is a band sum of two parallel copies of $E_0$ from the observation of the previous paragraph.
	
	Now we will use an induction argument.
	Suppose that the lemma holds for the path $P^{n_j}=P_0-P_1-\cdots-P_{n_j}$ ($0\leq n_j \leq n$) where $P^{n_j}$ contains at least one $P_l$ for $0\leq l\leq n_j$ which belongs to an E-cluster.
	(The subsequence $P^j$ of $P$ for every $0\leq j\leq n$ must be of minimal length in the sense of connecting the initial and terminal weak reducing pairs of $P^j$, otherwise the assumption of the original sequence cannot hold.)
	
	Let $n_{j+1}$ be the smallest integer such that some  $P_\alpha\subset P^{n_{j+1}}$ belongs to an E-cluster for $n_j<\alpha \leq n$.
	Let $E'_\alpha$ be the E-disk of the initial weak reducing pair of $P_\alpha$ and $E''_\alpha$ be that of the terminal weak reducing pair of $P_\alpha$.
	Since $E'_\alpha$ comes from the last $E$-disk of $P^{n_j}$, the induction hypothesis forces $E'_\alpha$ to be isotopic to $E_0$ in $W$ (if it is non-separating in $W$) or a band sum of two parallel copies of $E_0$ (if it is separating in $W$.)

	Consider the case $|P_\alpha|=1$.
	If $E'_\alpha$ is isotopic to $E_0$, then $E''_\alpha$ must be a band-sum of two parallel copies of $E_0$ by Lemma \ref{corollary-corollary-pants}.
	If $E'_\alpha$ is isotopic to a band sum of two parallel copies of $E_0$, then $E''_\alpha$ is non-separating in $W$, i.e. $E''_\alpha$ is the E-disk of the center of the E-cluster containing $P_\alpha$.
	But Lemma \ref{corollary-corollary-pants} forces $E''_\alpha$ to be determined uniquely by $E'_\alpha$.
	Moreover, since $E'_\alpha$ comes from the last E-disk of $P^{n_j}$, $E''_\alpha$ must be the same as the E-disk of the center of the directly previous E-cluster. 
	That is, $E''_\alpha$ is isotopic to $E_0$.

	Consider the case $|P_\alpha|=2$.
	Let $\beta$ be the largest integer such that $P_\beta$ is contained in some E-cluster and $P_\beta\cap P^{n_j}\neq \emptyset$.

	Let $\bar{E}_\beta$ be the E-disk of the terminal weak reducing pair of $P_\beta$ and $\tilde{E}_\beta$ is the E-disk of the center of the E-cluster containing $P_\beta$.
	Then we get $\bar{E}_\beta=E_\alpha'$ by the choice of $\alpha$ and $\beta$.
	Therefore, $\bar{E}_\beta$ is separating in $W$, i.e. $\bar{E}_\beta$ is a band sum of two parallel copies of $\tilde{E}_\beta$.
	But the induction hypothesis forces $\tilde{E}_\beta$ to be isotopic to $E_0$.
	Since the $E$-disk of the center of the E-cluster for $P_\alpha$ is determined uniquely up to isotopy by $E'_\alpha=\bar{E}_\beta$ by Lemma \ref{corollary-corollary-pants}, the E-disk of the center of the E-cluster containing $P_\alpha$ is also isotopic to $\tilde{E}_\beta=E_0$.
	Therefore, we conclude that $E''_\alpha$ is also a band sum of two parallel copies of $E_0$ by Lemma \ref{corollary-corollary-pants}.
	This completes the proof of the induction argument.

	Since $E'$ is the same as the E-disk of terminal weak reducing pair of $P^{n_j}$ for the largest $n_j$, this completes the proof.
\end{proof}


Now we give the proof of Theorem \ref{the-main-theorem}.
	Let $D_0$, $E_0$, $D_1$, and $E_1$ be the disks in the assumption of Theorem \ref{the-main-theorem}.
	Lemma \ref{key-lemma} means that $(D_0, E_0)$ and $(D_1, E_1)$ are not weakly compatible.
	Therefore, if we use Proposition \ref{prop-Bachman}, then the proof ends.\\
	
Now we give the proof of Corollary \ref{corollary-tunnelnumber-two}.
	Since the tunnel number is two and the genus of the splitting is three,  the splitting is unstabilized.
	Let $V$ be the compression body which meets the boundary of the knot exterior.
	Since $D_0$ and $D_1$ come from the cocores of $N(t_1)$ and $N(t_2)$, they are non-separating in $V$.	
	Obviously, they are not isotopic in $V$.
	Therefore, the splitting is critical by Theorem \ref{the-main-theorem}.
	The last statement is obtained directly from Theorem 5.1 of \cite{Bachman1}.
	This completes the proof.

Note that the definition of \textit{``critical surface''} of \cite{Bachman2} is significantly simpler and slightly weaker, than the one given in \cite{Bachman1}.
In other words, anything that was considered critical in \cite{Bachman1} is considered critical here as well.

\section{A connected sum of two 2-bridge knots induces a critical Heegaard splitting of its exterior\label{section-2-bridges}}


\begin{definition}[See Section 2 of \cite{LustigMoriah}]
	Let $K$ be a knot in $S^3$.
	A $2n$-plat projection (see \cite{BurdeZieschang}) gives rise to two \emph{canonical Heegaard splittings} of $S^3 - N(K)$ obtained as follows: 
	Consider first the system of arcs $\rho_1$, $\cdots$, $\rho_{n-1}$ which connect adjacent bottom bridges of $K$ (the bottom tunnels) as indicated Figure \ref{fig-wide}.
	One defines the compression body $V$ to be the union of a collar of $\partial N(K)$ and a regular neighborhood of $\rho_1$, $\cdots$, $\rho_{n-1}$.
	The handlebody $W$ is the complement of $V$ in $S^3-N(K)$. 
	The other Heegaard splitting is defined analogously by using the top tunnel system $\tau_1$, $\cdots$, $\tau_{n-1}$.
\end{definition}

\begin{figure}
	\includegraphics[width=8cm]{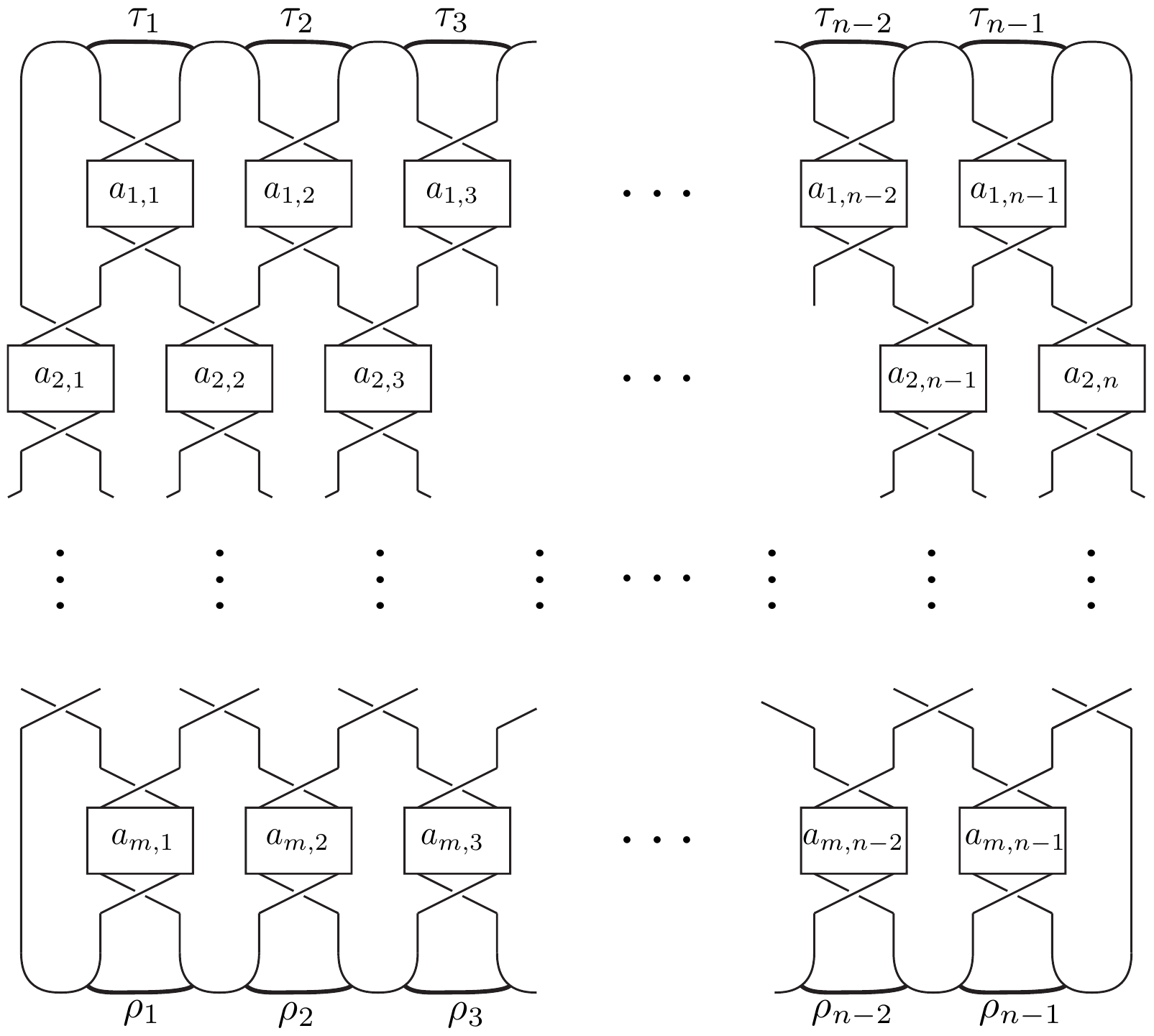}
	\caption{a 2n-plat diagram of $K$. \label{fig-wide} (This is Fig 2 of \cite{LustigMoriah}.)}
\end{figure}

\begin{definition}[Definition 2.1 of \cite{LustigMoriah}]$~$\\[-3ex]
	\begin{enumerate}
		\item A $2n$-braid will be called \emph{wide} if in its standard projection (i.e. every crossing is replaced by a node) there is no monotonically descending path connecting the top of the second strand to the bottom of the $(2n-1)$st, or vice versa. 
		\item A $2n$-plat projection of a knot or link will be called \emph{wide} if the underlying $2n$-braid is wide. 
		\item A knot or link $K\subset S^3$ will be called \emph{wide} if it has a wide $2n$-plat projection so that the corresponding canonical Heegaard splittings are irreducible.
	\end{enumerate}
\end{definition} 

Let $K_1$ and $K_2$ be two $2$-bridge knots.
Consider two standard diagrams of $K_1$ and $K_2$ where the diagram of $K_1$ has $m_1$ 2-braids and that of $K_2$ has $m_2$ 2-braids. 
Let $m=\max(m_1, m_2)$.
We can assume that both diagrams have $m$ 2-braids in a formal sense if we add trivial braids to the shorter diagram.
Isotope two diagrams so that these diagrams are depicted with respect to the same hight function.
In particular, we can assume that the fourth strand of $K_1$ is a vertical path in its projection and so is the first strand of $K_2$.
If we connect the forth strand of $K_1$ and the first strand of $K_2$ as in Figure \ref{fig-k1ck2}, then we get a connected sum $K_1\# K_2$.

\begin{figure}
	\includegraphics[width=6cm]{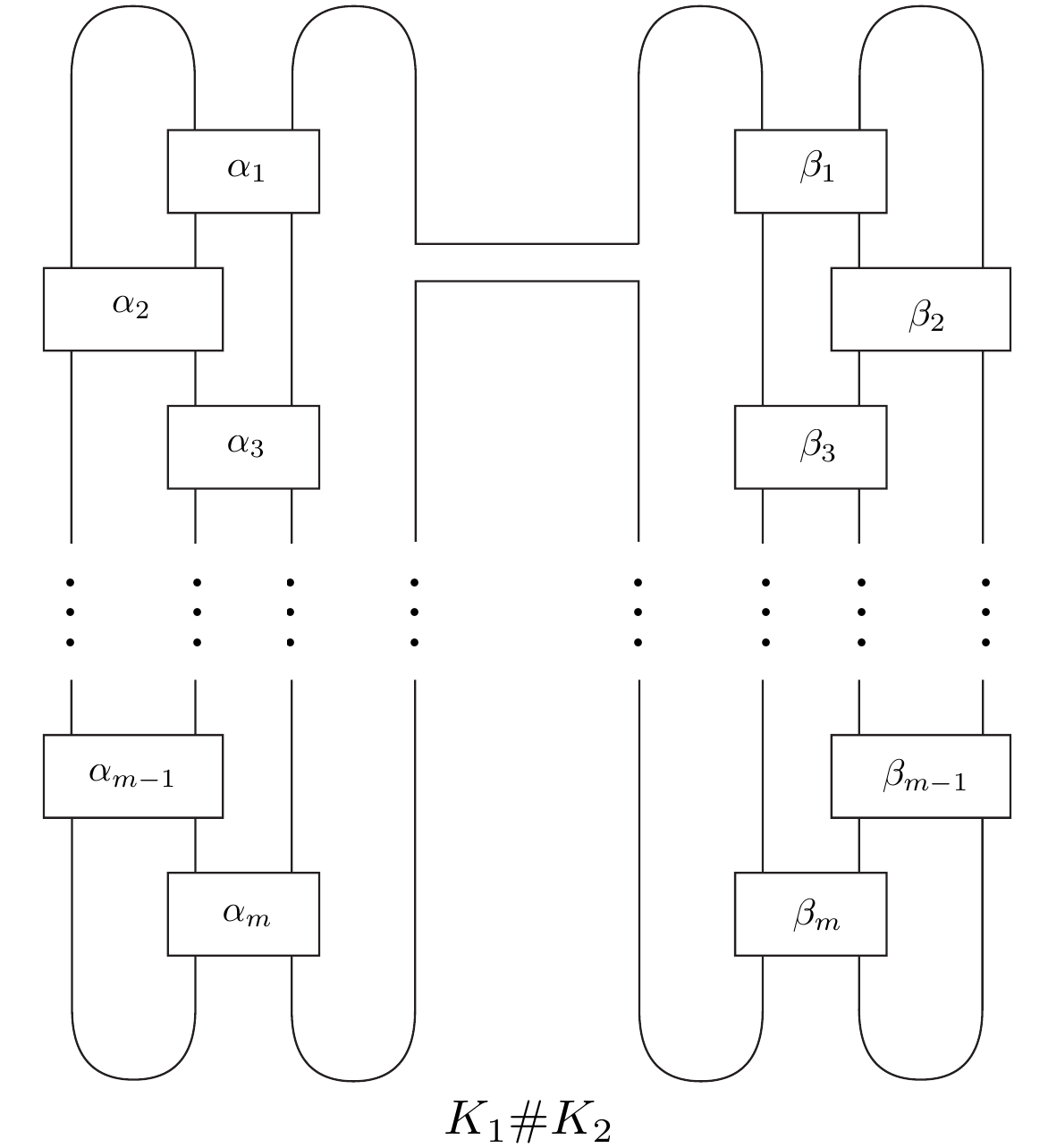}
	\caption{the connected sum of $K_1$ and $K_2$\label{fig-k1ck2}}
\end{figure}
Moreover, if we isotope it as in Figure \ref{fig-6plat}, then we get a $6$-plat projection of $K_1\# K_2$.
Note that $a_{k, 2}=0$ for even $k$ in this $6$-plat projection.

\begin{figure}
	\includegraphics[width=4cm]{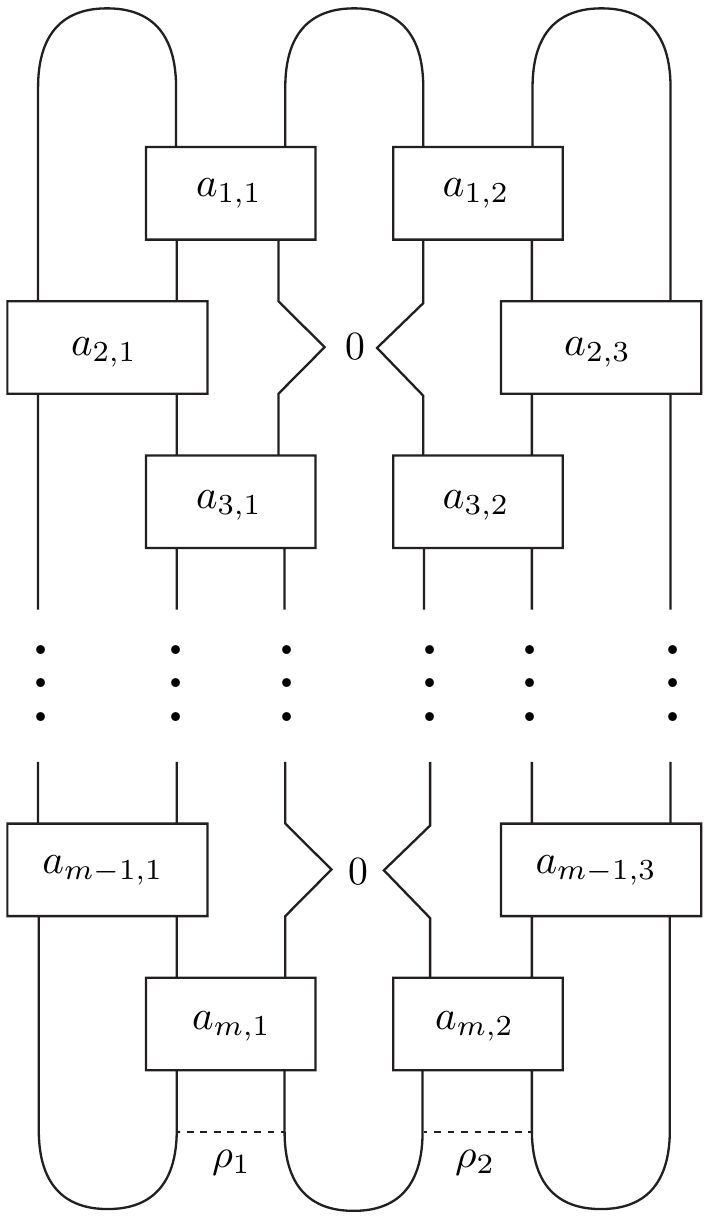}
	\caption{the 6-plat projection of $K_1\# K_2$\label{fig-6plat}}
\end{figure}

Let $T$ be the boundary of the exterior of $K_1\# K_2$.
If we consider two lower tunnels $\rho_1$ and $\rho_2$ for $K_1$ and $K_2$ respectively, then $N(T\cup\rho_1\cup\rho_2)$ gives a compression body $V$ and its complement gives a handlebody $W$, i.e. $\{\rho_1, \rho_2\}$ gives a tunnel system of $K_1\# K_2$.
Hence, we get a Heegaard splitting of the exterior of $K_1\#K_2$.
We call this splitting \emph{the Heegaard splitting of the exterior of $K_1\# K_2$ induced by the lower tunnels of $K_1$ and $K_2$}. 
Since $a_{k,2}=0$ for even $k$, it is easy to see that this $6$-plat projection of $K_1\#K_2$ is wide.

\begin{lemma}
	If $K_1$ and $K_2$ are two-bridge knots, then $K_1\# K_2$ is a wide knot.
\end{lemma}

\begin{proof}
	Since we get a wide $6$-plat projection for $K_1\# K_2$, it is sufficient to show that two canonical Heegaard splittings from the $6$-plat projection are irreducible.
 
	Suppose that any one of both splittings is stabilized.
	This means that $K_1\# K_2$ is tunnel number one knot.
	But every tunnel number one knot must be prime (see \cite{MorimotoSakuma} and \cite{Scharlemann1}), this gives a contradiction.
	Therefore, these Heegaard splittings are unstabilized.
	Moreover, these splitting are irreducible since a reducible splitting forces the manifold to be reducible or the splitting to be stabilized (see 3.4.1 of \cite{SaitoScharlemannSchultens}.)
	This completes the proof.
\end{proof}

	If we use Proposition 2.2 of \cite{LustigMoriah}, then we get the canonical splittings are weakly reducible.
	The idea of finding weak reducing pairs by Lustig and Moriah is described in the proof of the next theorem.

\begin{figure}
	\centering
	\includegraphics[width=5cm]{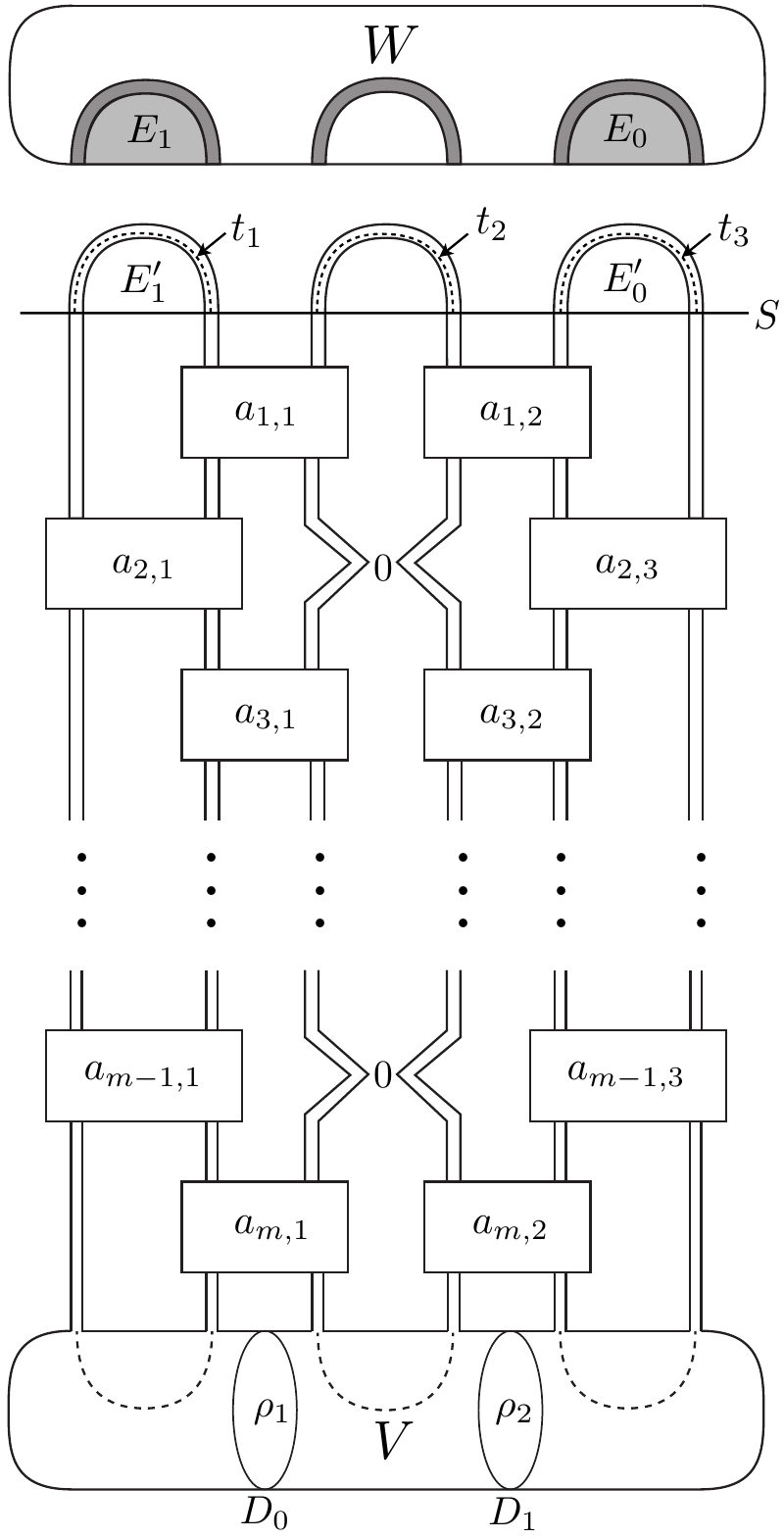}
	\caption{the tunnel system of $K_1\# K_2$ and the induced Heegaard splitting of the exterior
	\label{fig-Heegaard}}
\end{figure}

\begin{theorem}
	Let $K_1$ and $K_2$ be $2$-bridge knots.
	Then the Heegaard splitting of the exterior of $K_1\# K_2$ induced by the lower tunnels of $K_1$ and $K_2$ is critical.
\end{theorem}

\begin{proof}
	First, we will find two weak reducing pairs by using the arguments of the proof of Proposition 2.2 of \cite{LustigMoriah}.

	Consider an equatorial $2$-sphere $S$ intersecting $K_1\# K_2$  just below the top bridges and cutting off a $3$-ball $B$ with $3$ unknotted arcs $t_1$, $\cdots$, $t_3$ as indicated in Figure \ref{fig-Heegaard}. 
	The handlebody $W$ is ambient isotopic to the handlebody $W'=B-N(\cup t_i)$, and this isotopy $h_t$ is given by moving the above equatorial $2$-sphere $S=h_0(S)$ monotonically down the braid to a level $h_1(S)$ just above the bottom bridges, through horizontal $2$-spheres.
	In the 6-plat projection, we can find two essential  disks $E'_0$ and $E'_1$ in $W'$ as Figure \ref{fig-Heegaard}.
	Let $D_0$ and $D_1$ be the cocore disks of $N(\rho_1)$ and $N(\rho_2)$ respectively and $E_i=h_1(E'_i)$ for $i=0,1$.
	Each $E_i$ is an essential non-separating disk in $W$ for $i=0,1$.
	Moreover, $D_i\cap E_i=\emptyset$ for $i=0,1$ since $a_{k,2}=0$ for even $k$. 
	Hence, we get the two weak reducing pairs $(D_0, E_0)$ and $(D_1, E_1)$.
	But $D_0$ and $D_1$ are non-separating in $V$ and $D_0$ is not isotopic to $D_1$ in $V$,  this splitting is critical by using Corollary \ref{corollary-tunnelnumber-two}.
\end{proof}

\section{A connected sum of two $(1,1)$-knots induces a critical Heegaard splittings of its exterior \label{section-1-1-knots}}

\begin{definition}
	We say that a curve on a hanldebody $H$ is \textit{primitive} if there is an essential disk in the
handlebody intersecting the curve in a single point.
	An annulus $A$ on a handlebody $H$ is \textit{primitive} if
its core curve is primitive.
%
%
	Note that a curve on a handlebody $H$ is \textit{primitive} if it represents a primitive element in the free group $\pi_1(H)$.
\end{definition}

\begin{definition} (Section 2 of \cite{Moriah2004})
Let $K$ be a connected sum of two knots $K_1$ and $K_2$.
For a given Heegaard splitting $(V_1,V_2)$ for $S^3-N(K)$, we will choose a decomposing annulus $A$ which intersects the compression body $V_1$ in two spanning annuli $A_1^\ast$, $A_2^\ast$ and a minimal collection of disks $\mathcal{D} = \{D_1,\, \cdots,\,D_l\}$.
We can easily see that if we evacuate one side of $A$ from $S^3-N(K)$ corresponding to the nontrivial tangle of $K_2$, then the remaining part is homeomorphic to $E(K_1)$, or vise versa. 
Therefore, we can identify $E(K_i)$ and the corresponding remaining part for $i=1,2$.
When we cut $E(K)$ along $A$, the Heegaard splitting $(V_1, V_2)$ induces Heegaard splittings on both $E(K_1)$ and $E(K_2)$ as follows.
	\begin{enumerate}
		\item $V_1^i=(V_1\cap E(K_i))\cup_{\mathcal{D}\cup A_1^\ast \cup A_2^\ast} (N(A)\cap E(K_i))$, where each $V_1^i$ is a compression body for $i=1,2$.
		\item $V_2^1$ and $V_2^2$ are the two components of $V_2-N(A)$ where each $V_2^i$ is a handlebody for $i=1,2$ and $V_1^i\cap V_2^j=\emptyset$ for $i\neq j$.
		\item Therefore, $(V_1^i, V_2^i)$ for $i=1,2$ is a Heegaard splitting of $E(K_i)$ for $i=1,2$.
	\end{enumerate}
Note that $V_1^1\cap V_1^2 = A$.

Conversely, suppose that there is a Heegaard splitting $(V_1^i, V_2^i)$ of $E(K_i)$ for $i=1,2$ and assume that each $V_1^i$ meets $\partial N(K_i)$ for $i=1,2$.
Choose one meridional annulus in $\partial N(K_1)$ and another one in $\partial N(K_2)$ and identify both annuli.
Then we get the exterior of a connected sum of $K_1$ and $K_2$ as in Figure \ref{fig-connected_sum}, where the decomposing annulus $A$ comes from the identified one.
Assume that the attaching disks for the $1$-handles from each tunnel system of $K_i$ realizing the Heegaard splitting $(V_1^i, V_2^i)$ miss $N(A)$ for $i=1,2$ up to isotopy, including some handle slides.
Since the attaching disks miss $N(A)=A\times I$, we can assume that the $A\times\{0\}$ is a subsurface of $\partial_+ V_1^1$ and $A\times \{1\}$ is also that of $\partial_+ V_1^2$. 

If we remove $N(A)$ from $E(K_i)$, then we get the projection image of the decomposing annulus, say $A_i$, in the boundary of $E(K_i)-N(A)$. 
Then, $A_i$ can be written as
	$$A_i=A_1^{\ast i}\cup A_2^{\ast i}\cup A^{i},$$
	where $A_1^{\ast i}$ and $A_2^{\ast i}$ are two spanning annuli in $E(K_i)$ and $A_i - \operatorname{int}(A_1^{\ast i}\cup A_2^{\ast i})$ is a meridional annulus $A^i$ in $\partial_+ V_1^i = \partial_+ V_2^i$ for $i=1,2$ (if we project $A_i$ into $A$ conversely, then we can regard it as a meridional one, see Figure \ref{fig-connected_sum}.)
	
\begin{figure}
	\centering
	\includegraphics[width=10cm]{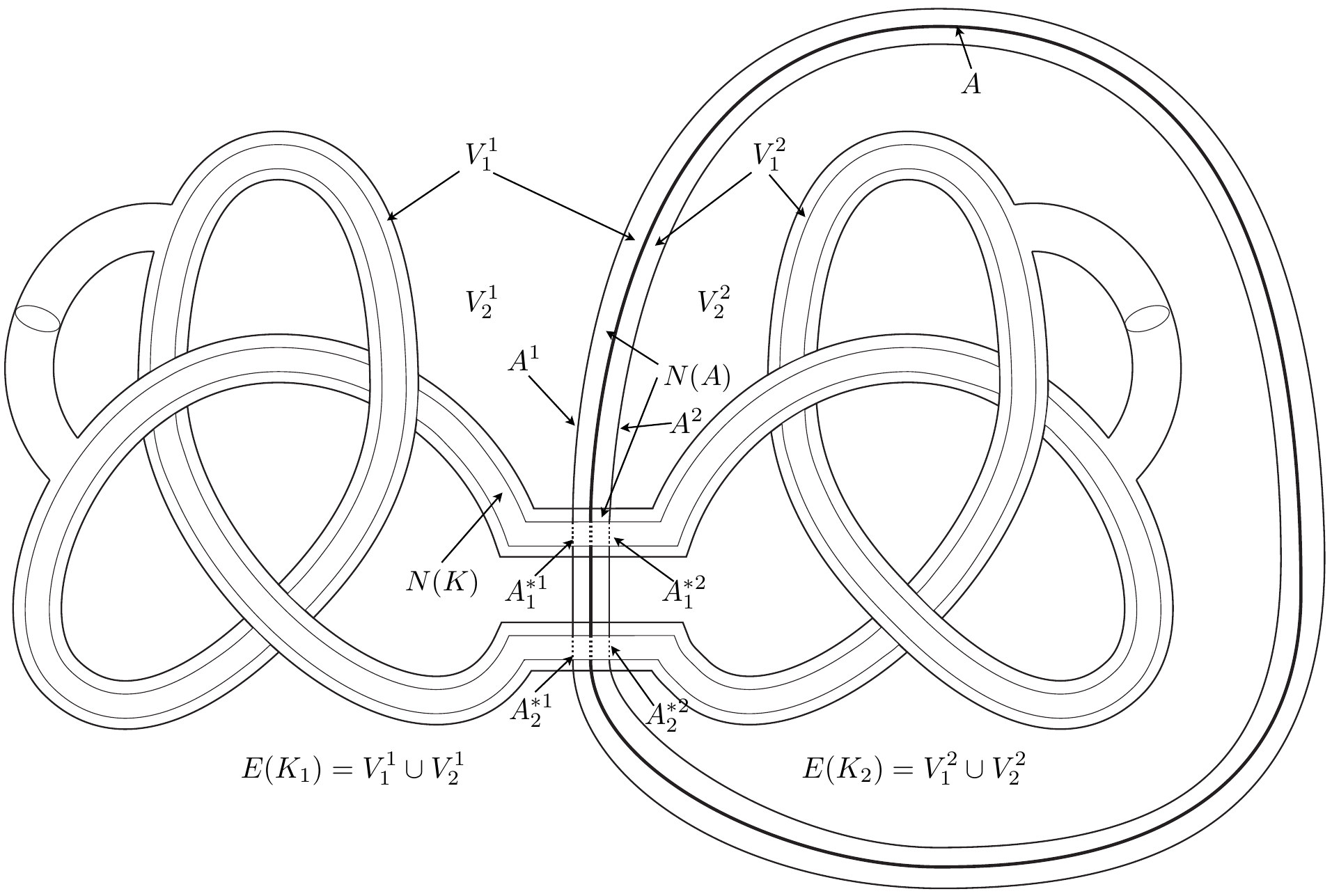}
	\caption{The decomposition of $E(K)$ into $E(K_1)\cup_{A} E(K_2)$ when $\mathcal{D}=\emptyset$.
	\label{fig-connected_sum}}
\end{figure}

	If we remove the $N(A)$ part from $V_1^i$ for $i=1,2$ and glue $V_1^1$ and $V_1^2$ by identifying $A_1^{\ast 1}= A_1^{\ast 2}$ and $A_2^{\ast 1}= A_2^{\ast 2}$, then we get a solid $V_1$.
	Also we glue $V_2^1$ and $V_2^2$ by identifying $A^1$ and $A^2$ and get a solid $V_2$.
	$V_1$ is a compression body since $V_1$ is obtained from $\partial N(K)\times I$ by adding $1$-handles corresponding to the tunnel systems of $K_1$ and $K_2$.
	But $V_2$ is a handlebody if and only if $A^1$ or $A^2$ is a primitive annulus on $V_2^1$ or $V_2^2$ respectively.
	In the case that $V_2$ becomes a handlebody, we will say $(V_1, V_2)$ is \textit{the induced Heegaard splitting} of $E(K)$ induced by $(V_1^i, V_2^i)$ for $i=1,2$ and we get $t(K)\leq t(K_1)+t(K_2)$.
	Moreover, if the tunnel number of $K_i$ is one for each $i=1,2$, then the genus of the induced Heegaard splitting is three.
\end{definition}
	
\begin{theorem}\label{theorem-main-B}
	Let $K_1$ and $K_2$ be tunnel number one knots, $K$ be a connected sum of $K_1$ and $K_2$, and $(V_1^i, V_2^i)$ be a Heegaard splitting of genus two of $E(K_i)$ for $i=1,2$.
	Assume $A^1$ and $A^2$ as in the previous arguments.
	If each $A^i$ is primitive on $V_2^i$ for $i=1,2$, then the induced Heegaard splitting $(V_1, V_2)$ of $E(K)$ is critical.
\end{theorem}

\begin{proof}
	Since each $A^i$ is primitive on $V_2^i$ for $i=1,2$, we can find an essential disk $E_2^i\subset V_2^i$ for $i=1,2$ such that $A^i\cap E_2^i$ is an essential arc in $A^i$.
	Let us consider two parallel copies $E'_i$ and $E''_i$ of $E_2^i$ as in Figure \ref{fig-E}.
	The core circle of $A^i$ is divided by $\partial E'_i\cup \partial E''_i$ into two arcs $\gamma$ and $\delta$ such that $\partial \gamma=\partial \delta.$
	If we consider the band sums of $E'_i$ and $E''_i$ along $\gamma$ and $\delta$ in $V_2^i$, then at least one must be an essential separating disk in $V_2^i$ since $g(\partial V_2^i)=2$.

	Let this band sum be $\bar{E}^i$ (see Figure \ref{fig-E}.)

\begin{figure}
	\centering
	\includegraphics[width=8cm]{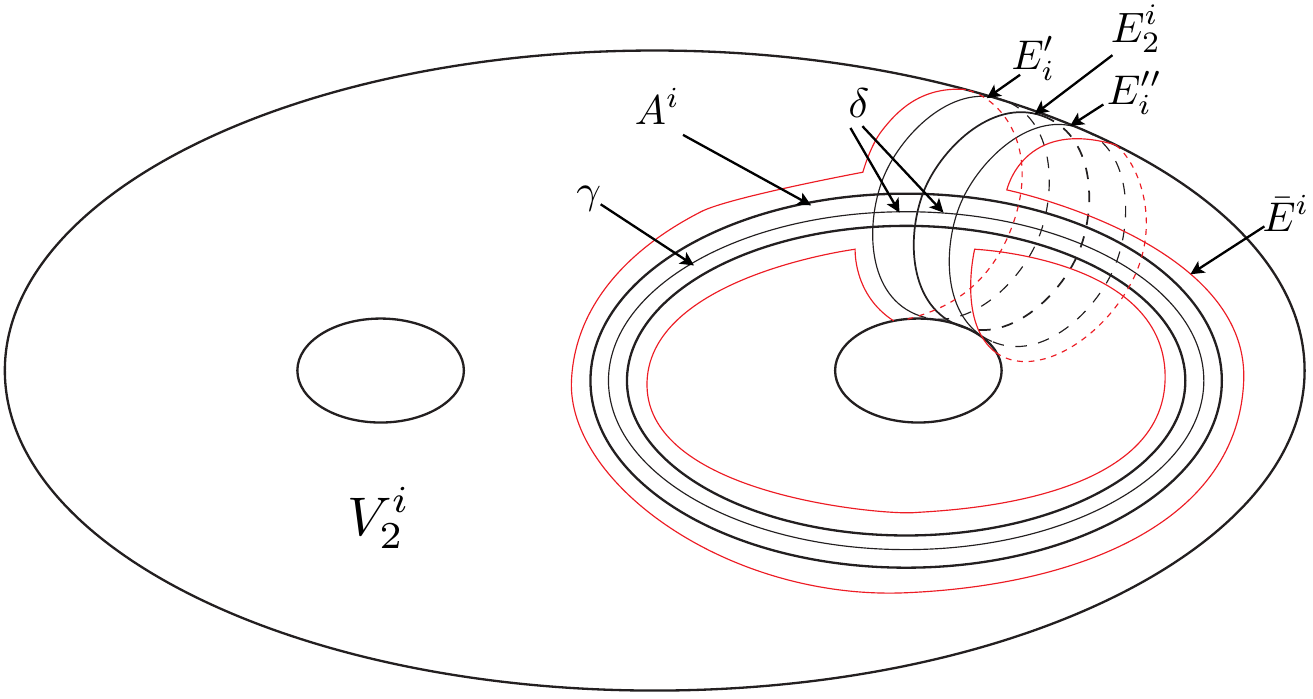}
	\caption{the bandsum $\bar{E}^i$ \label{fig-E}}
\end{figure}

	We can isotope $\bar{E}^i$ in $V_2^i$ so that $\bar{E}^i\cap A^i=\emptyset$.
	Here, each $\bar{E}^i$ is an essential separating disk in $V_2^i$ for $i=1,2$, and so is in $V_2$.
	(We can check that $\bar{E}^1$ and $\bar{E}^2$ cut $V_2$ into three solid tori $V_2'$, $V_2''$, and $V_2'''$, where $V_2'$ ($V_2'''$ resp.) is determined by only $\bar{E}^1$ ($\bar{E}^2$ resp.) and $V_2''$ is determined by both $\bar{E}^1$ and $\bar{E}^2$, i.e. $V_2''$ contains the identified image of the primitive annuli $A^1$ and $A^2$.)
	Let $D_i\subset V_1^i$ be the cocore disk of the tunnel from the Heegaard splitting of $E(K_i)$.
	It is obvious that $D_i\cap \bar{E}^j=\emptyset$ for $i\neq j$ since each $D_i$ also misses $A_i$ for $i=1,2$.
	
	Therefore, we get the two weak reducing pairs $(D_1, \bar{E}^2)$ and $(D_2, \bar{E}^1)$.

	Since the induced Heegaard splitting implies that $t(K_1\# K_2)\leq 2$ and a composite knot cannot be of tunnel number one, we get $t(K_1\# K_2) =2$, i.e. the induced splitting must be unstabilized of genus three.
	Therefore, the induced Heegaard splitting $(V_1, V_2)$ is critical by Corollary \ref{corollary-tunnelnumber-two}.
\end{proof}

In Proposition 2.1 of \cite{Morimoto2000}, Morimoto proved that a tunnel number one knot $K$ is a $(1,1)$-knot if and only if there is a genus two Heegaard splitting $(V_1, V_2)$ in its exterior such that there exist a spanning annulus $A$ in $V_1$ and a compressing disk $D$ in $V_2$ such that  $A$ meets $\partial N(K)$ in a meridian of $K$ and meets $D$ on $\partial_+V_1 = \partial_+V_2$ in a single point transversely.
By using this, we induce the following corollary.

\begin{corollary}
	Let $K_1$ and $K_2$ be two $(1,1)$-knots.
	We can realize a connected sum of $K_1$ and $K_2$ so that there exists a critical Heegaard splitting in its exterior.
\end{corollary}

\begin{proof}
	Let $(V_1^i, V_2^i)$ be a Heegaard splitting of $E(K_i)$ realizing the tunnel number of $K_i$ and assume that $V_1^i$ meets $\partial N(K_i)$ for $i=1,2$.
	Let $\bar{A}^i\subset V_1^i$ be the spanning annulus from Morimoto's proposition and $m^i$ be the meridian $\bar{A}^i\cap N(K_i)$ for $i=1,2$.
	If we consider a small neighborhood of $m^i$ in $\partial N(K_i)$, then we get a meridional annulus $A_i$ on $\partial N(K_i)$ for $i=1,2$.
	So if we realize the connected sum by identifying $A_1$ and $A_2$, i.e. the decomposing annulus $A$ is the identified image of these annuli in the exterior of the connected sum, then we can use the arguments finding the weak reducing pairs in the proof of Theorem \ref{theorem-main-B}, and we get the induced splitting is critical.
\end{proof}

Note that we can skip Section \ref{section-2-bridges} if we only want to check the existence of a critical Heegaard splitting in the exterior since a $2$-bridge knot is also a kind of $(1,1)$-knot.
But the choice of weak reducing pairs of Section \ref{section-2-bridges} is different from that of Section \ref{section-1-1-knots}.
Indeed, the weak reducing pairs in Section \ref{section-2-bridges} consist of only non-separating disks in their compression bodies.
But we take separating disks in $V_2$ for the weak reducing pairs in Section \ref{section-1-1-knots}.

\appendix 
\section{an equivalent condition for a weak reducing pair to be determined uniquely by a compressing disk\label{section-unique-wr}}

In this section, we will prove the following theorem.

\begin{theorem}\label{theorem-main}
	Suppose $F$ is an unstabilized, genus three Heegaard surface in a closed, orientable, irreducible 3-manifold.
	Then either $F$ is the amalgamation of two genus $2$ Heegaard surfaces over a torus, or every compressing disk for F belongs to at most one weak reducing pair.
\end{theorem}

First, we introduce the following lemma.

\begin{lemma}[Corollary 3.10 of \cite{JungsooKim1}]\label{corollary-difffourcurves}
	Let $M$ be an irreducible $3$-manifold and $H=(V,W;F)$ be an unstabilized genus three Heegaard splitting of $M$. 
	There are two adjacent $2$-subsimplices $f_1$ and $f_2$ in $\DVW$ such that $f_1$ is a D-face representing the $1$-compatible weak reducing pairs $(\bar{D}, \tilde{E})$ and $(\tilde{D}, \tilde{E})$ and $f_2$ is an E-face representing the $1$-compatible weak reducing pairs $(\tilde{D}, \tilde{E})$ and  $(\tilde{D}, \bar{E})$ if and only if the disks $\bar{D}$, $\tilde{D}\subset V$ and $\bar{E}$, $\tilde{E}\subset W$ hold the following conditions.
 		\begin{enumerate}
			\item \label{corollary-difffourcurves-1} Four boundary curves of the disks represent different isotopy classes in $F$.
			\item \label{corollary-difffourcurves-2}One of $\partial \bar{D}$ and $\partial \tilde{D}$ ($\partial \bar{E}$ and $\partial \tilde{E}$ resp.) is separating and the other is non-separating in $F$.	
			\item \label{corollary-difffourcurves-4} $\partial\bar{D}\cup \partial \tilde{D}$ cuts off a pair of pants from $F$, and so does $\partial\bar{E}\cup \partial \tilde{E}$. 
			Moreover, both pairs of pants are disjoint in $F$ if the four disks are all disjoint.
			\item \label{corollary-difffourcurves-3} Either (A) the four disks are all disjoint, or (B) we can replace one separating disk among these disks by another separating disk in the same compression body so that the resulting four disks are all disjoint and satisfying the conditions (\ref{corollary-difffourcurves-1}), (\ref{corollary-difffourcurves-2}), and (\ref{corollary-difffourcurves-4}).
		\end{enumerate}
\end{lemma}

We can imagine a situation as in Figure \ref{fig-1-3} if there exist four disks satisfying the four conditions of Lemma \ref{corollary-difffourcurves} and they are all disjoint. The thick curves are separating and the thin curves are non-separating in $F$.

\begin{figure}
	\includegraphics[width=8cm]{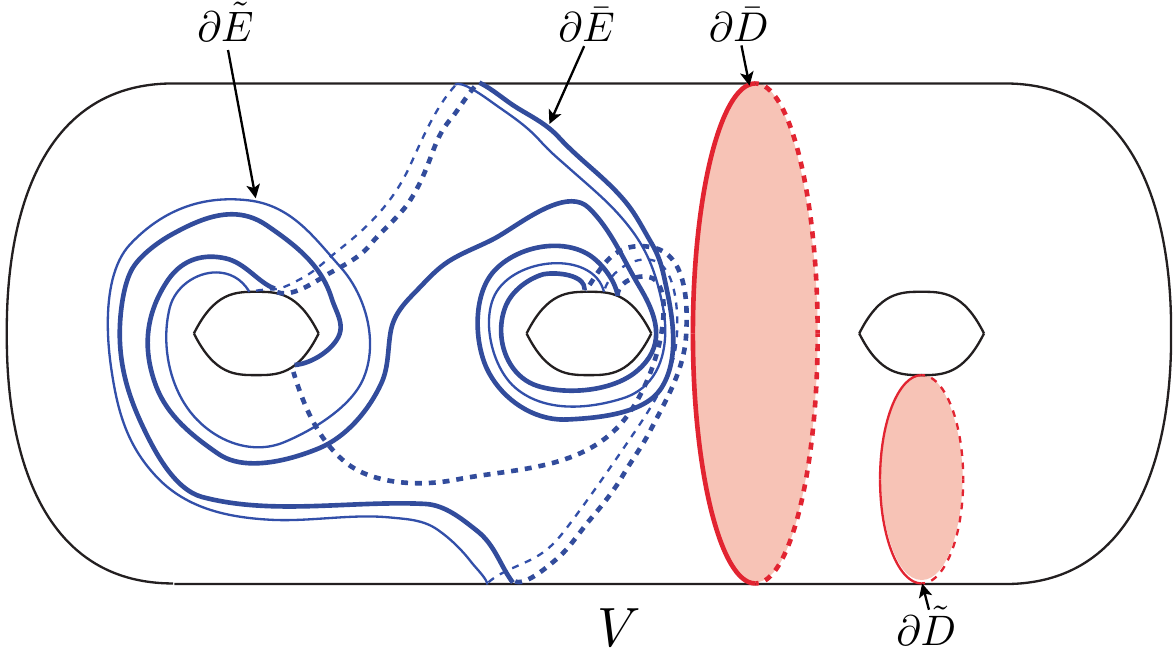}
	\caption{An example of the boundaries of four disks satisfying the four conditions of Lemma \ref{corollary-difffourcurves}\label{fig-1-3}}
\end{figure}

The following lemma guarantees a D-face and an E-face sharing a weak reducing pair when we can choose $1$-compatible weak reducing pairs.

\begin{lemma}[Corollary 4.3 of \cite{JungsooKim1}] \label{corollary-differentlabel}
	Assume $M$, $H$, and $F$ as in Lemma \ref{corollary-pants}, and add the assumption that $M$ is closed.
	If we can choose $1$-compatible weak reducing pairs, then $\DVW$ must have a D-face and an E-face such that one shares a weak reducing pair with the other.
\end{lemma}

Now we consider some equivalent conditions when we can choose only a $0$-compatible weak reducing pair when $M$ is closed. 

\begin{lemma}[Corollary 4.4 of \cite{JungsooKim1}]\label{corollary-last}
	Assume $M$, $H$, and $F$ as in Lemma \ref{corollary-pants}, and add the assumption that $M$ is closed.
	Then the following three statements are equivalent.
	\begin{enumerate}
		\item \label{cm1} There is no weak reducing pair such that the disks are separating in their handlebodies.
		\item \label{cm2} $H$ cannot be represented as an amalgamation of genus two splittings along a torus.
		\item \label{cm3} We cannot choose $1$- or more compatible weak reducing pairs.
	\end{enumerate}
\end{lemma}

Now we prove Theorem \ref{theorem-main}.
Suppose that $F$ is not an amalgamation of two genus $2$ surfaces along a torus.
If $F$ is strongly irreducible, then we get the conclusion.
Assume that $F$ is weakly reducible and consider a weak reducing pair $(D,E)$ containing a compressing disk $D$. 
If there is another weak reducing pair $(D, E')$ containing $D$, then we can choose $1$-compatible weak reducing pairs for $F$ by Lemma \ref{lemma-Bachman}.
But this contradicts Lemma \ref{corollary-last}.
Therefore, the weak reducing pair $(D,E)$ is the unique one containing the compressing disk $D$.

Conversely, suppose that $F$ is an amalgamation of two genus $2$ surfaces along a torus.
Then we can easily choose $1$-compatible weak reducing pairs.
(In Figure \ref{fig-amalgamation}, we can find $1$-compatible weak reducing pairs $(D, E)$ and $(D, E')$.)
That is, there is a compressing disk contained in two or more weak reducing pairs.\\
This completes the proof of Theorem \ref{theorem-main}.\\

\begin{figure}
	\includegraphics[width=12cm]{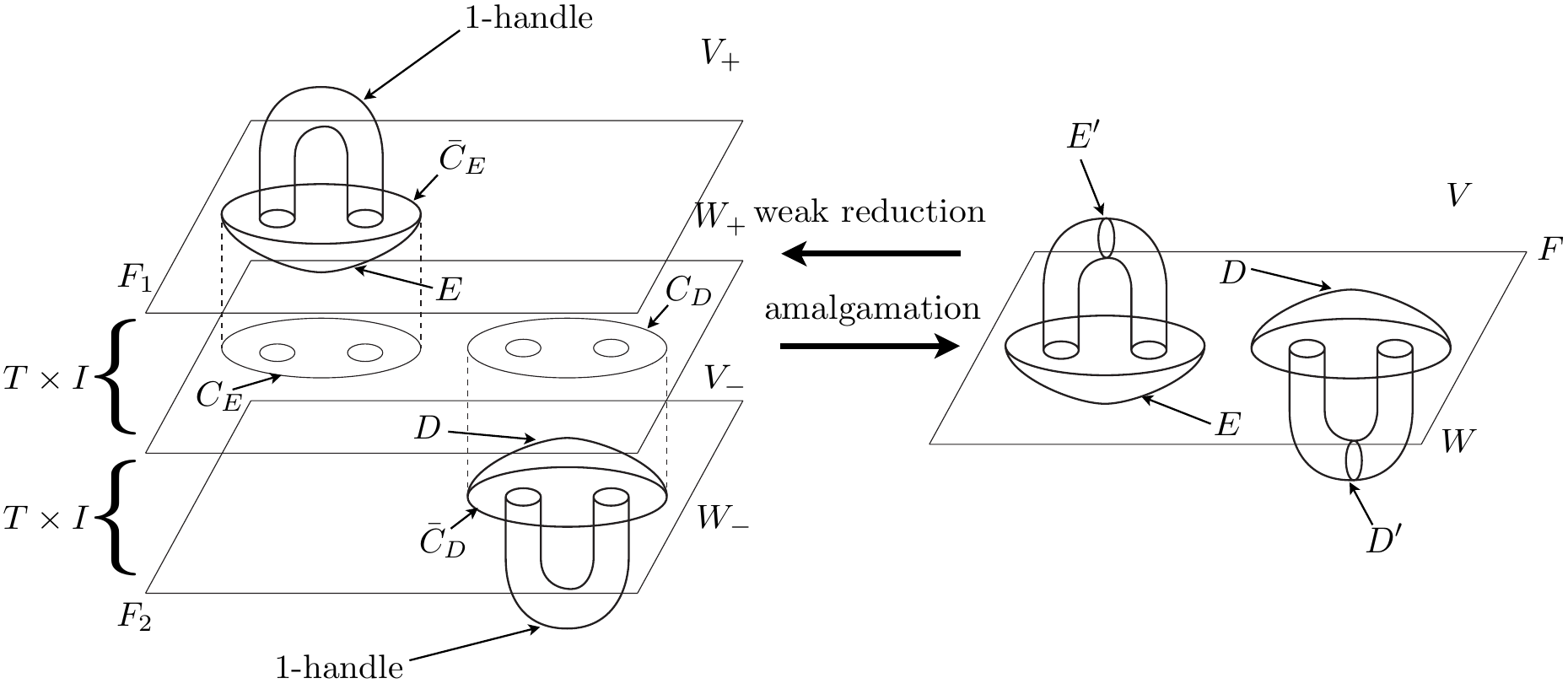}
	\caption{If we amalgamate two thick surfaces along a torus, then we can find a weak reducing pair such that the disks are separating.\label{fig-amalgamation}}
\end{figure}

In the remaining part of this section, we give an interpretation of Theorem \ref{theorem-main}.

Assume $M$, $H$, and $F$ as in Lemma \ref{corollary-pants}, and add the assumption that $M$ is closed.
In the case when we can choose $1$-compatible weak reducing pairs, we get a D-face and an E-face in $\DVW$ such that one shares a weak reducing pair with the other by Corollary \ref{corollary-differentlabel}.
Moreover, we can take the disjoint four disks $\bar{D}$, $\tilde{D}$, $\bar{E}$, and $\tilde{E}$ satisfying the four conditions of Corollary \ref{corollary-difffourcurves}.
That is, we can take a weak reducing pair $(\tilde{D}, \tilde{E})$ which consists of only non-separating disks.
Moreover, if there is a weak reducing pair $(D,E)$ such that $D$ is separating in $V$, then we can take a non-separating disk $D'$ from the meridian disk of the solid torus $V'$ obtained by cutting $V$ along $D$, and we can guarantee $D'\cap E=\emptyset$ by Lemma 3.1 of \cite{JungsooKim1}.
This means that there exist $1$-compatible weak reducing pairs $(D,E)$ and $(D',E)$ sharing $E$.
Hence, if we cannot choose $1$-compatible weak reducing pairs, then every weak reducing pair must consists of only non-separating disks.
In summary, we can choose a weak reducing pair which consists of non-separating disks in any case.

If we untelescope a genus three, unstabilized Heegaard splitting $H=(V,W;F)$ of a closed, irreducible $3$-manifold by a weak reducing pair $(D,E)$ which consists of non-separating disks, then we get the following two cases after getting rid of some impossible cases for irreducible manifolds,

	\begin{enumerate}[(a)]
		\item The thin surface is a torus, i.e. $\partial D\cup \partial E$ does not separates $F$ into several pieces.
		\item The thin surface consists of two tori, i.e. $\partial D\cup \partial E$ cuts $F$ into two twice-punctured tori.
	\end{enumerate}

Therefore, Theorem \ref{theorem-main} means that either (a) for some compressing disk, we can choose two or more weak reducing pairs containing it or (b) for every compressing disk, we can choose at most one weak reducing pair containing it.
Moreover, this also means that a pair of thin tori is uniquely determined up to isotopy by only one compressing disk (as well as by the weak reducing pair) in the case (b) (see figure \ref{fig-twocases}.)
Note that each of $D$ and $E$ is the $1$-handle connecting the neighborhood of two tori from the minus boundary of its compression body in the case (b) (consider the standard genus three splitting of $T^3$ and a weak reducing pair of this splitting.) In Lemma 6 of \cite{JJ}, Johnson proved that a weak reducing pair is determined uniquely by a compressing disk for the standard genus three splitting of $T^3$ and Theorem \ref{theorem-main} is the generalization of the Johnson's Lemma.

\begin{figure}
	\includegraphics[width=11cm]{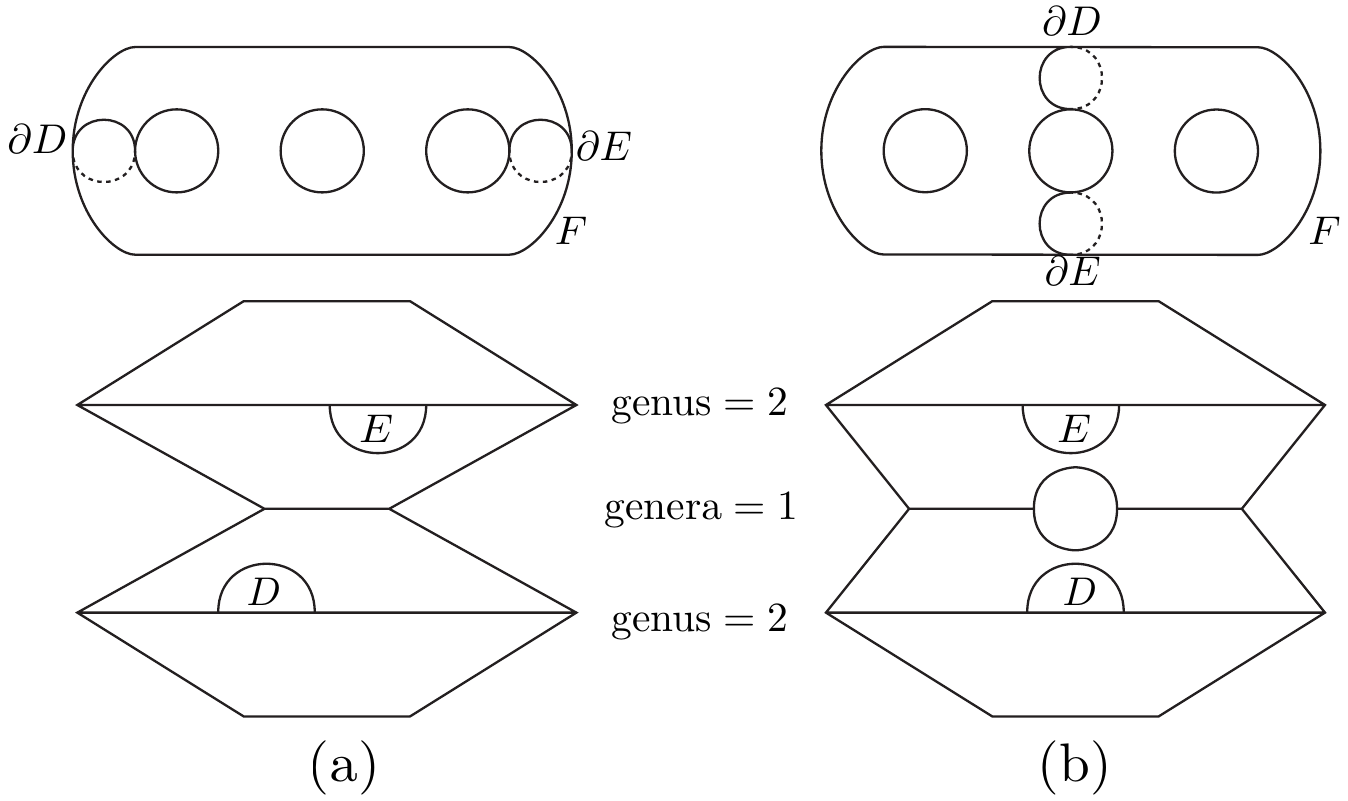}
	\caption{The two cases of Theorem \ref{theorem-main}\label{fig-twocases}}
\end{figure}

\section*{Acknowledgments}
	The author is grateful to Prof. Jung-Hoon Lee for helpful conversations.
	The author is also grateful to someone for suggesting Theorem \ref{theorem-main} to me as a corollary of the result of \cite{JungsooKim1}.
	The author was supported by the National Research Foundation of Korea (NRF) grant funded by the Korean Government (2010-0019516).

\bibliographystyle{agsm}

\end{document}